\documentclass[11pt]{article}
\usepackage[tbtags]{amsmath}
\usepackage{amssymb}
\usepackage{amsthm}
\usepackage[misc]{ifsym}
\usepackage{cases}
\usepackage{mathrsfs}

\numberwithin{equation}{section}

\normalsize

\title{\bf  Relationship between Maximum Principle and Dynamic Programming Principle for Risk-Sensitive Stochastic Optimal Control Problems with Applications \thanks{This work is financially supported by the National Key R\&D Program of China (2022YFA1006104), National Natural Science Foundations of China (12471419, 12271304), and Shandong Provincial Natural Science Foundations (ZR2024ZD35, ZR2022JQ01).}}
\author{\normalsize Huanqing Dong\thanks{\it School of Mathematics, Shandong University, Jinan 250100, P.R. China, E-mail: donghuanqing@mail.sdu.edu.cn},\quad Jingtao Shi\thanks{\it Corresponding author, School of Mathematics, Shandong University, Jinan 250100, P.R. China, E-mail: shijingtao@sdu.edu.cn}}

\newtheorem{mythm}{Theorem}[section]

\newtheorem{mylem}{Lemma}[section]
\newtheorem{Remark}{Remark}[section]

\begin{document}

\maketitle

\noindent{\bf Abstract:}\quad
This paper is concerned with the relationship between maximum principle and dynamic programming principle for risk-sensitive stochastic optimal control problems. Under the smooth assumption of the value function, relations among the adjoint processes, the generalized Hamiltonian function, and the value function are given. As an application, a linear-quadratic risk-sensitive portfolio optimization problem in the financial market is discussed.

\vspace{2mm}
	
	
\vspace{2mm}
	

\section{Introduction}

It is well-known that Pontryagin's \emph{maximum principle} (MP) and Bellman's \emph{dynamic programming principle} (DPP) are two important tools for solving optimal control problems. So a natural question is: Are there any connections between these two methods? The relationship between MP and DPP is studied in many literatures. For the deterministic optimal control problems, the results were discussed by Fleming and Rishel \cite{FR75}, Barron and Jensen \cite{BJ86} and Zhou \cite{Z90-1}. Bensoussan \cite{B82} (see also \cite{YZ99}) obtained the relationship between MP and DPP for stochastic optimal control problems. Zhou \cite{Z90-2} generalized their relationship using the viscosity solution theory (see also \cite{Z91, YZ99}), without the assumption that the value function is smooth.

In this paper, we consider the relationship between MP and DPP for a risk-sensitive stochastic optimal control problem, where the system is given by the following controlled \emph{stochastic differential equation} (SDE):
\begin{equation}\label{state SDE}
\left\{
\begin{aligned}
		dX(t) &= f(t, X(t), u(t)) dt + \sigma(t, X(t), u(t)) dW(t),\\
		X(0) &= x,
\end{aligned}
\right.
\end{equation}
and the cost/objective functional is defined by
\begin{equation}\label{risk-sensitive cost functional}
	J(u(\cdot)) = \mu^{-1}\log \mathbb{E} \left[\exp\left\{\mu\int_0^T l(t, X(t), u(t))  dt+ \mu g(X(T))\right\}\right].
\end{equation}
Risk-sensitive optimal control problems has attracted many researchers since the early work of Jacobson \cite{J73}. On the one hand, this problem can be reduced to the risk-neutral case under the limit of the risk-sensitive parameter. On the other hand, it has close connection with differential games \cite{J92,LZ01}, robust control theory \cite{LZ01,MB17}, and can be applied to describe the risk attitude of an investor by the risk-sensitive parameter in mathematical finance \cite{BP99,FS00,KN02}.

Whittle \cite{W90,W91} first derived Pontryagin's MP based on the theory of large deviation for a risk-sensitive stochastic optimal control problem. A new risk-sensitive maximum principle was established by Lim and Zhou \cite{LZ05} under the condition that the value function is smooth, based on the general stochastic MP of Peng \cite{P90} and the relationship between MP and DPP of Yong and Zhou \cite{YZ99}. This idea is used in Shi and Wu \cite{SW11,SW12} for studying jump processes case and mathematical finance. Djehiche et al. \cite{DTT15} extends the results of \cite{LZ05} to risk-sensitive control problems for dynamics that are non-Markovian and mean-field type. Further developments on risk-sensitive stochastic control problems in a diverse set of settings, see  Moon et al. \cite{MDB19}, Chala \cite{C17-1,C17-2}, Khallout and Chala \cite{KC19}, Moon \cite{M20}, Lin and Shi \cite{LS23,LS25} and the references therein.

In El Karoui and Hamedene \cite{EH03}, the risk-sensitive objective functional is characterized by a {\it backward SDE} (BSDE) with quadratic growth coefficient that does not necessarily satisfy the usual Lipschitz condition. Consider the following BSDE with quadratic growth coefficient:
\begin{equation}\label{BSDE quadratic growth}
\left\{
\begin{aligned}
	dY(t) &= -\left[l(t, X(t), Y(t), Z(t), u(t)) + \frac{\mu}{2}|Z(t)|^2\right]dt +  Z(t)dW(t),\\
	Y(T) &= g(X(T)).
\end{aligned}
\right.
\end{equation}
By the DPP approach, Moon \cite{M21} extended the results of classical risk-sensitive optimal control problem studied in \cite{FS06} of the objective functional given by the BSDE like (1.3). They showed that the value function is the unique viscosity solution of the corresponding Hamilton-Jacobi-Bellman equation under an additional risk-sensitive condition. When $l$ is independent of $y$ and $z$, it is easy to check that $Y(0)=\mu^{-1}\log \mathbb{E} \left[\exp\left\{\mu\int_{0}^{T} l(t, X(t), u(t))  \, dt+ \mu g(X(T))\right\}\right]$. Thus, the risk-sensitive optimal control problem of minimizing $J(u(\cdot))$ in (1.2) subject to (1.1) is equivalent to minimizing $Y(0)$ subject to (1.1) and (1.3). In this case, the controlled system becomes a decoupled \emph{forward-backward SDE} (FBSDE), with the objective functional of a BSDE with quadratic growth coefficient.

For the relationship between MP and DPP for optimal control problems of decoupled FBSDEs/recursice utilities, Shi and Yu \cite{S13} investigated the local case in which the control domain is convex and the value function is smooth. Nie et al. \cite{NSW17} studied the general case by the second-order adjoint equation of Hu \cite{Hu17}, within the framework of viscosity solution. Hu et al. \cite{HJX20} studied the relationship between general MP and DPP for the fully coupled forward-backward stochastic controlled system within the framework of viscosity solutions. Note that all the aforementioned references considered stochastic optimal control problems for BSDE objective functionals with linear growth coefficients satisfying Lipschitz conditions. For the relationship between MP and DPP for optimal control problems of decoupled FBSDEs and BSDE objective functionals with quadratic growth coefficients, very recently, Wu et al. \cite{W25} established the connection between MP and DPP for the risk-sensitive recursive utility singular control problem under the assumption of smooth value function, where the cost functional is given by a BSDE with quadratic growth, driven by a discontinuous semimartingale.

Inspired by the above works, in this paper, we will derive the relationship between MP and DPP for risk-sensitive stochastic optimal control problems. For this problem, because the cost functional is given by \eqref{BSDE quadratic growth}, we wish to connect the MP of Moon et al. \cite{MDB19} under certain conditions. Moreover, since the cost functional (1.2) can be defined by the controlled BSDE with quadratic growth coefficient, we wish to apply the DPP of Moon \cite{M21} under certain conditions. Specifically, when the value function is smooth, we give relations among the adjoint processes, the generalized Hamiltonian function, and the value function.

The rest of this paper is organized as follows. In Section 2, we state our risk-sensitive stochastic optimal control problem and give some preliminaries. In Section 3, we show the relationship between MP and DPP for the problem, by the martingale representation technique. In Section 4, we apply the theoretical results to a portfolio optimization problem in a factor model setting. Specifically, we obtain the state feedback optimal control by the risk-sensitive MP and the DPP, and the relations we obtained are illustrated explicitly. Finally, in Section 5, we give the concluding remarks.

$\mathit{Notations}$. In this paper, we denote by $\mathbb{R}^n$ the space of $n$-dimensional Euclidean space, by $\mathbb{R}^{n \times d}$ the space of $n \times d$ matrices, and by $\mathcal{S}^n $ the space of $n \times n$ symmetric matrices. For $x \in \mathbb{R}^n$ , $x^\top$ denotes its transpose. \textit{I} denotes the identity matrix.  $\langle \cdot , \cdot \rangle$ and $| \cdot |$ denote the scalar product and norm in the Euclidean space, respectively. The trace of a matrix A is $\operatorname{tr}(A)$. Let $C([0,T] \times \mathbb{R}^n)$ be the set of real-valued continuous functions defined on $[0,T] \times \mathbb{R}^n$, $C^{1,k}([0,T] \times \mathbb{R}^n)$ ($k \geq 1$) be the set of real-valued functions such that $f \in C([0,T] \times \mathbb{R}^n)$, $\frac{\partial f}{\partial t}$ and $\frac{\partial^k f}{\partial x^k}$ are continuous and bounded.

\section{Problem statement and preliminaries}

Let $(\Omega, \mathcal{F}, \mathbf{P})$ be a complete probability space equipped with a $d$-dimensional standard Brownian motion $\{W_t\}_{t \geq 0}$. For fixed $t \geq 0$, the natural filtration $\{\mathcal{F}^t_s\}_{s \geq t}$ is generated as $\mathcal{F}_s=\sigma\{W_r-W_t;t \leq r \leq s\} \bigvee \mathcal{N}$, where $\mathcal{N}$ contains all $\mathbf{P}$-null sets in $\mathcal{F}$. In particular, if $t=0$, we write $\mathcal{F}_s=\mathcal{F}^t_s$. For $t\geq0$, we denote by $L^2_{\mathcal{F}}(t,T;\mathbb{R}^n)$ the the space of $\mathbb{R}^n$-valued $\{\mathcal{F}^t_s\}_{s \geq t}$-adapted processes on $[t,T]$ with $\mathbb{E}\int_t^T |x(s)|^2 ds < \infty$, and by $\mathcal{S}(t,T;\mathbb{R}^n)$ the space of $\mathbb{R}^n$-valued $\{\mathcal{F}^t_s\}_{s \geq t}$-adapted processes on $[t,T]$ with $\mathbb{E}\big[\sup_{s \in [t,T]}|x(s)|^2 \big]< \infty$.

Let $T>0$ be finite and let $U \subset \mathbb{R}^m$ is a compact metric space. For any initial time and state $(t,x) \in [0,T) \times \mathbb{R}^n$, consider the state $X^{t,x;u}_\cdot \in \mathbb{R}^n$ given by the following controlled $\mathrm{SDE}$:
\begin{equation}\label{controlled state SDE}
\left\{
\begin{aligned}
	dX^{t, x; u}_s &= f(s, X^{t, x; u}_s, u_s) ds + \sigma(s, X^{t, x; u}_s, u_s) dW_s,\\
	X^{t, x; u}_t &= x.
\end{aligned}
\right.
\end{equation}
Here $f: [0,T] \times \mathbb{R}^n \times U \to \mathbb{R}^n$, $\sigma: [0,T] \times \mathbb{R}^n \times U \to \mathbb{R}^{n \times d}$ are given functions.

Given $t \in [0,T)$, the set of admissible controls are defined as follows:
\begin{equation*}
	\mathcal{U}[t,T]:=\left\{u:[0,T] \times \Omega \to U | u_\cdot \in L^2_{\mathcal{F}}(t,T;U)\right\}.
\end{equation*}
For $u_\cdot \in \mathcal{U}[t,T]$ and $x \in \mathbb{R}^n$, an $\mathbb{R}^n$-valued process $X^{t,x;u}_\cdot$ is called a solution to \eqref{controlled state SDE} if it is an $\{\mathcal{F}^t_s\}_{s \geq t}$-adapted process such that \eqref{controlled state SDE} hold. We refer to such $(u_\cdot,X^{t,x;u}_\cdot)$ as an admissible pair.

We make the following assumption.\\
$(\mathbf{H1})$  $f,\sigma$ are uniformly continuous in $s$, and there exists a constant $C>0$ such that for any $s \in [0,T]$, $x,\hat{x} \in \mathbb{R}^n,u,\hat{u} \in U$,
\begin{equation}\label{assumption H1}
	\begin{cases}
	|f(s,x,u)-f(s,\hat{x},\hat{u})| + |\sigma(s,x,u)-\sigma(s,\hat{x},\hat{u})| \leq C(|x-\hat{x}|+|u-\hat{u}|),\\
	|f(s,0,u)|\leq C.
	\end{cases}
\end{equation}
For any $u_\cdot \in \mathcal{U}[t,T]$, under (H1), the equation \eqref{controlled state SDE} has a unique solution $X^{t, x; u}_\cdot \in \mathcal{S}(t,T;\mathbb{R}^n)$ by the classical $\mathrm{SDE}$ theory.
We consider the following cost functional with initial condition  $(t,x) \in [0,T) \times \mathbb{R}^n$ and $u_\cdot \in \mathcal{U}[t,T]$:
\begin{equation}\label{risk-sensitive cost functional 2}
	J(t,x;u_\cdot ) = \mu^{-1}\log \mathbb{E} \left[\exp\left\{\mu\int_t^T l(s,X^{t, x; u}_s,u_s) ds+ \mu g(X^{t, x; u}_T)\right\}\right],
\end{equation}
where  $l: [0,T] \times \mathbb{R}^n \times U \to \mathbb{R}$, $g: [0,T] \times \mathbb{R}^n \to \mathbb{R}$, and $\mu>0$ is the risk-sensitive parameter that captures robustness (Fleming and Soner \cite{FS06}). Let $J_1 \equiv \int_0^T l(s,X_s^{0, x; u},u_s) ds+ g(X_T^{0, x; u})$. When the risk-sensitive parameter $\mu$ is small, the cost functional \eqref{risk-sensitive cost functional 2} can be expanded as
\begin{equation*}
	\mathbb{E}[J_1] + \frac{\mu}{2}\operatorname{Var}(J_1) + O(\mu^2),
\end{equation*}
where $\operatorname{Var}(J_1)$ denote the variance of $J_1$. If $\mu<0$, which implies that the controller is risk-seeking from an economic point of view.  When $\mu \to 0$, it is reduced to risk-neutral case.

Our risk-sensitive stochastic optimal control problem is the following.

{\bf Problem (SR)}. For given $(t,x) \in [0,T) \times \mathbb{R}^n$, to minimize \eqref{risk-sensitive cost functional 2} subject to \eqref{controlled state SDE} over $\mathcal{U}[t,T]$.

From \cite{M21} (see also \cite{EH03}), we note that minimizing \eqref{risk-sensitive cost functional 2} in Problem (SR) is equivalent to minimizing the following cost functional
\begin{equation}\label{recursive utility}
	J(t,x;u_\cdot) = Y^{t,x;u}_{t},
\end{equation}
where $Y_\cdot^{t,x;u}$ is the first component of the solution pair $(Y^{t,x;u}_\cdot,Z^{t,x;u}_\cdot)$ to the scalar-valued BSDE:
\begin{equation}\label{scalar-valued BSDE}
\left\{
\begin{aligned}
	dY^{t, x; u}_s &= -\left[l(s, X^{t, x; u}_s, u_s) + \frac{\mu}{2}|Z^{t,x;u}_s|^2\right]ds +  Z^{t,x;u}_sdW_s,\\
	Y^{t, x; u}_T &= g(X^{t, x; u}_T).
\end{aligned}
\right.
\end{equation}
In fact, let $\widehat{Y}^{t,x;u}_s:=\exp\big\{\mu Y^{t, x; u}_{s}\big\}$, and note that $Y^{t, x; u}_s = \mu^{-1}\log\widehat{Y}^{t,x;u}_s$. Applying It\^o's formula to $\widehat{Y}^{t,x;u}_\cdot$, we have
\begin{equation*}
\left\{
\begin{aligned}
	d\widehat{Y}^{t,x;u}_s &= \mu \widehat{Y}^{t,x;u}_s\left[-l(s, X^{t, x; u}_s, u_s)ds +  Z^{t,x;u}_sdW_s\right],\\
	\widehat{Y}^{t,x;u}_T &= \exp\big\{\mu g(X^{t, x; u}_T)\big\}.
\end{aligned}
\right.
\end{equation*}
This is a linear BSDE, its explicit solution can be obtained:
\begin{equation*}
		\widehat{Y}^{t,x;u}_s= \mathbb{E} \left[\exp\left\{\mu\int_t^T l(s,X^{t, x; u}_s,u_s) ds+ \mu g(X^{t, x; u}_T)\right\}|\mathcal{F}_t\right].
\end{equation*}
Hence
\begin{equation}\label{explicit solution of linear BSDE}
\begin{aligned}
	J(0,x;u_\cdot) &= Y_0^{0,x;u} = \mu^{-1}\log\widehat{Y}_0^{0,x;u}\\
	           &= \mu^{-1}\log\mathbb{E} \left[\exp\left\{\mu\int_0^T l(s,X_s^{0, x; u},u_s) ds+ \mu g(X_T^{0, x; u})\right\}\right].\\
\end{aligned}
\end{equation}
Note that \eqref{explicit solution of linear BSDE} corresponds to the classical risk-sensitive optimal control problem (see \cite{FS06}, \cite{MDB19}).

We introduce the following assumption.\\
$(\mathbf{H2})$ $l$ is uniformly continuous in $s$, $g$ is uniformly continuous in $x$ and bounded. For $\phi=l,g$, there exists a constant $L>0$, such that for any $t \in [0,T], x,\hat{x} \in \mathbb{R}^n$, and $u,\hat{u} \in U$,
\begin{equation}\label{assumption H2}
\begin{cases}
	|\phi(s,x,u)-\phi(s,\hat{x},\hat{u})| \leq L(|x-\hat{x}|+|u-\hat{u}|),\\
	|\phi(s,0,u)|\leq L.
\end{cases}
\end{equation}

Note that (H2) implies $|l(s, x, u)+\frac{\mu}{2}|z|^2| \leq L(1 + |x| + |z|^2)$.
Thus \eqref{scalar-valued BSDE} can be viewed as a controlled BSDE with quadratic growth coefficient studied in \cite{K00,EH03,Z17}. For any $u_\cdot \in \mathcal{U}[t,T]$ and given unique solution $X_\cdot^{t,x,u}$ to (2.1), under (H2), BSDE \eqref{scalar-valued BSDE} admits a unique solution $(Y_\cdot^{t,x,u},Z_\cdot^{t,x,u}) \in \mathcal{S}(t,T;\mathbb{R}) \times L^2_{\mathcal{F}}(t,T;\mathbb{R}^{1 \times d})$ (see \cite{M21,EH03,Z17}).

We define the value function as
\begin{equation}\label{value function}
	V(t,x):=\inf \limits_{u_{\cdot} \in\, \mathcal{U}[t,T]}J(t,x;u_{\cdot}),\quad(t,x) \in [0,T) \times \mathbb{R}^n.
\end{equation}
Any $\bar{u}_\cdot \in \mathcal{U}[t,T]$ satisfied \eqref{value function} is called an optimal control, and the corresponding solution $\bar{X}_\cdot^{t,x;\bar{u}}$ to \eqref{controlled state SDE} is called optimal state trajectory.

\begin{Remark}\label{V is meaningful}
From \cite{M21}, we know that under (H1) and (H2), the above value function is a deterministic function, so our definition \eqref{value function} is meaningful.
\end{Remark}

In the following, we introduce the risk-sensitive DPP and MP approaches for $\mathbf{Problem(SR)}$ in the literatures, respectively. Moon \cite{M21} studied a generalized risk-sensitive optimal control problem, where the objective functional is defined by a controlled BSDE with quadratic growth coefficient. A generalized risk-sensitive DPP was obtained. And the corresponding value function is a viscosity solution to the corresponding Hamilton-Jacobi-Bellman equation. Under some additional parameter conditions, the viscosity solution is unique. When $l$ is independent of $y$ and $z$ in (3) of Moon \cite{M21}, the BSDE in (3) is reduced to the BSDE in \eqref{scalar-valued BSDE}. Thus we can state the HJB equation for our Problem (SR):
\begin{equation}\label{HJB equation}
\begin{cases}
	\frac{\partial V}{\partial t}(t,x) + \inf\limits_{u \in U} G\left(t,x,u,\frac{\partial V}{\partial x}(t,x),\frac{\partial^2 V}{\partial x^2}(t,x)\right)=0,\quad(t,x) \in [0,T) \times \mathbb{R}^n,\\
	V(T,x)=g(x),\quad x \in \mathbb{R}^n,
\end{cases}
\end{equation}
where the generalized Hamiltonian function $G:[0,T] \times \mathbb{R}^n \times U \times \mathbb{R}^n \times \mathcal{S}^n \to \mathbb{R}$ is defined by
\begin{equation}\label{generalized Hamiltonian}
\begin{aligned}
	G(t,x,u,p,P)&:= l(t,x,u) + \langle p ,f(t,x,u) \rangle  + \frac{\mu}{2}\big|\sigma(t,x,u)^\top p\big|^2 \\
	&\quad+ \frac{1}{2}\operatorname{tr}\big(\sigma(t,x,u)\sigma(t,x,u)^\top P\big).	
\end{aligned}	
\end{equation}

The following result can be inferred by their Theorem 2 of \cite{M21}.
\begin{mylem}\label{V is solution to HJB equation}
Let (H1), (H2) hold and $(t,x) \in [0,T) \times \mathbb{R}^n$. Suppose $V(\cdot,\cdot) \in C^{1,2}([0,T] \times \mathbb{R}^n)$. Then $V(\cdot,\cdot)$ is a solution to the HJB equation in \eqref{HJB equation}.
\end{mylem}

Moon et al. \cite{MDB19} studied a two-player risk-sensitive zero-sum differential game, where the drift and diffusion terms in the controlled state SDE are dependent on the state and controls of both players, and the objective functional is the risk-sensitive type. A stochastic MP type necessary condition was established (their Theorem 1 in \cite{MDB19}). If we consider single-player case, that is, let $f$ is independent of $v$ in (1) of \cite{MDB19}, so that the state equation is consistent with \eqref{controlled state SDE}. Moreover, let $\frac{1}{\delta} = \mu$ and $l$ is independent of $v$ in (4) of \cite{MDB19} so that the cost functional is consistent with \eqref{risk-sensitive cost functional 2}. Thus we can obtain the MP of our Problem (SR).

First, we need introduce the following assumption.\\
$(\mathbf{H3})$ $f,\sigma,l,g$ are twice continuously differentiable with respect to $x$. Their first and second-order partial derivatives in $x$ are bounded, Lipschitz continuous in $(x,u)$ with a constant $L$, and bounded by $L(1 + |x| + |u|)$.

Let $\bar{u}_\cdot$ be an optimal control, and $\bar{X}_\cdot^{t,x;\bar{u}}$ is corresponding optimal state trajectory. For any $s \in [0,T]$, we denote
\begin{equation*}
	\bar{f}(s):=f(s,\bar{X}_s^{t,x;\bar{u}},\bar{u}_s),\quad\bar{\sigma}(s):=\sigma(s,\bar{X}_s^{t,x;\bar{u}},\bar{u}_s),\quad\bar{l}(s) := l(s,\bar{X}_s^{t,x;\bar{u}},\bar{u}_s),
\end{equation*}
and similar notations used for all their derivatives. We also introduce the following notations:
\begin{equation*}
	\sigma(s,x,u)=(\sigma_1(s,x,u),\cdots,\sigma_d(s,x,u)),
\end{equation*}
where $\sigma_j : [0,T] \times \mathbb{R}^n \times U \to \mathbb{R}^n$ for $j = 1,\cdots,d$. Let $\frac{\partial \sigma_j}{\partial x}$ be the partial derivative of $\sigma_j$ with respect to $x$, which is an $n \times n$ dimensional matrix.

For given $(t,x) \in [0,T) \times \mathbb{R}^n$, since the value function is smooth, we introduce the following first-order adjoint process $(p_\cdot,q_\cdot)\equiv(p_\cdot,q_{1\cdot},\cdots,q_{d\cdot}) \in  L^2_{\mathcal{F}}(t,T;\mathbb{R}^n) \times L^2_{\mathcal{F}}(t,T;\mathbb{R}^n,\cdots,\mathbb{R}^n)$ satisfying
\begin{equation}\label{first-order adjoint equation}
\left\{
\begin{aligned}
	dp_s &= -\bigg[ \frac{\partial \bar f}{\partial x}(s)^\top p_s + \frac{\partial \bar l}{\partial x}(s) + \sum_{j=1}^d \frac{\partial \bar \sigma_j}{\partial x}(s)^\top q_{js}
	+ \mu\sum_{j=1}^d \bar{\sigma}_j(s)^\top p_s\frac{\partial \bar \sigma_j}{\partial x}(s)^\top p_s  \\
	&\qquad + \mu\sum_{j=1}^d q_{js}\bar{\sigma}_j(s)^\top p_s \bigg]ds + q_sdW_s, \\
	p_T &= \frac{\partial \bar g}{\partial x}(\bar{X}_T^{t,x;\bar{u}}),
\end{aligned}
\right.
\end{equation}
and the second-order adjoint process $(P_\cdot,Q_\cdot)\equiv(P_\cdot,Q_{1\cdot},\cdots,Q_{d\cdot}) \in L^2_{\mathcal{F}}(t,T;\mathcal{S}^n) \times L^2_{\mathcal{F}}(t,T;\mathcal{S}^n,\\\cdots,\mathcal{S}^n)$ satisfying
\begin{equation}\label{second-order adjoint equation}
\left\{
\begin{aligned}
	dP_s &= -\bigg[ \frac{\partial \bar f}{\partial x}(s)^\top P_s + P_s\frac{\partial \bar f}{\partial x}(s)
    + \sum_{j=1}^d \frac{\partial \bar \sigma_j}{\partial x}(s)^\top \left(P_s + \mu p_sp_s^\top\right)\frac{\partial \bar \sigma_j}{\partial x}(s)\\
	&\quad + \sum_{j=1}^d \frac{\partial \bar \sigma_j}{\partial x}(s)^\top\left(Q_{js} + \mu \bar{\sigma}_j(s)^\top p_sP_s+\mu p_sq_{js}^\top\right) \\
	&\quad + \sum_{j=1}^d \left(Q_{js} + \mu \bar{\sigma}_{j}(s)^\top p_sP_s+\mu q_{js}p_s^\top\right)\frac{\partial \bar \sigma_j}{\partial x}(s) \\
	&\quad + \mu\sum_{j=1}^d q_{js}q_{js}^\top + \frac{\partial^2 \bar H}{\partial x^2}(s) + \mu\sum_{j=1}^d \bar{\sigma}_{js}^\top p_sQ_{js} \bigg]ds
	+ \sum_{j=1}^dQ_{js}dW_{js}, \\
	P_T &= \frac{\partial^2 g}{\partial x^2}(\bar{X}_{T}^{t,x;\bar{u}}),
\end{aligned}
\right.
\end{equation}
where $\bar{H}(s):= H(s,\bar{X}_s^{t,x;\bar{u}},\bar{u}_s,p_s,q_s)$ with the Hamiltonian function $H$ defined by
\begin{equation}\label{risk-sensitive Hamiltonian}
\begin{aligned}
	H(s,x,u,p,q)&:= \langle p ,f(s,x,u) \rangle + l(s,x,u) + \sum_{j=1}^dq_j^\top\sigma_j(s,x,u) \\
	&\quad + \mu \sum_{j=1}^d\sigma_j(s,x,u)^\top p\bar{\sigma}_j(s)^\top p.
\end{aligned}	
\end{equation}

Then, we have the following MP for our Problem (SR).
\begin{mylem}\label{risk-sensitive MP}
Suppose that (H1)-(H3) hold and $(t,x) \in [0,T) \times \mathbb{R}^n$ be fixed. Let $(\bar{u}_\cdot,\bar{X}_\cdot^{t,x;\bar{u}})$ be an optimal pair for our Problem (SR). Then, there exist unique solutions
\begin{equation*}
\left\{
\begin{aligned}
	&(p_\cdot,q_{1\cdot},\cdots,q_{d\cdot}) \in  L^2_{\mathcal{F}}(t,T;\mathbb{R}^n) \times L^2_{\mathcal{F}}(t,T;\mathbb{R}^n,\cdots,\mathbb{R}^n),\\
	&(P_\cdot,Q_{1\cdot},\cdots,Q_{d\cdot}) \in L^2_{\mathcal{F}}(t,T;\mathcal{S}^n) \times L^2_{\mathcal{F}}(t,T;\mathcal{S}^n,\cdots,\mathcal{S}^n),
\end{aligned}
\right.
\end{equation*}
to the first-order adjoint equation \eqref{first-order adjoint equation} and the second-order adjoint equation \eqref{second-order adjoint equation}, respectively, such that the following variational inequality hold:
\begin{equation}\label{maximum condition-H}
\begin{aligned}
	&H(s,\bar{X}_s^{t,x;\bar{u}},u,p_s,q_s)-H(s,\bar{X}_s^{t,x;\bar{u}},\bar{u}_s,p_s,q_s) \\
	&+ \frac{1}{2}\operatorname{tr}\left[\left(\sigma(s,\bar{X}_s^{t,x;\bar{u}},u)-\bar{\sigma}(s)\right)^\top\left(P_s + \mu p_sp_s^\top\right)\left(\sigma(s,\bar{X}_s^{t,x;\bar{u}},u)-\bar{\sigma}(s)\right)\right] \geq 0,\\
	&\hspace{6cm}\forall u \in U, \quad \text{a.e.} s \in [t,T],\quad \mathbf{P}\text{-a.s.},
\end{aligned}	
\end{equation}
or, equivalently,
\begin{equation}\label{maximum condition-mathcal H}
	\mathcal{H}(s,\bar{X}_s^{t,x;\bar{u}},\bar{u}_s) \leq \mathcal{H}(s,\bar{X}_s^{t,x;\bar{u}},u),\quad \forall u \in U, \quad \text{a.e.} s \in [t,T],\quad \mathbf{P}\text{-a.s.},
\end{equation}
where
\begin{equation}\label{mathcal H}
\begin{aligned}
	\mathcal{H}(s,x,u)&:= H(s,x,u,p_s,q_s) - \frac{1}{2}\operatorname{tr}\left[\bar{\sigma}(s)^\top\left(P_s + \mu p_sp_s^\top\right)\bar{\sigma}(s)\right]\\
	&\quad+\frac{1}{2}\operatorname{tr}\left[\left(\sigma(s,x,u)-\bar{\sigma}(s)\right)^\top\left(P_s + \mu p_sp_s^\top\right)\left(\sigma(s,x,u)-\bar{\sigma}(s)\right)\right].
\end{aligned}	
\end{equation}
\end{mylem}

\section{Main result}

In this section, under the assumption that the value function is smooth, we investigate the relationship between the MP and DPP for $\mathbf{Problem(SR)}$. Specifically, we obtain the relationship among the value function $V$, the generalized Hamiltionian function $G$ and the adjoint processes. Our main result is the following.
\begin{mythm}\label{relation of MP and DPP}
Let (H1)-(H3) hold, and  $(t,x) \in [0,T) \times \mathbb{R}^n$ be fixed. Suppose that $\bar{u}_\cdot$ is an optimal control for Problem (SR), and $(\bar{X}_\cdot^{t,x;\bar{u}},\bar{Y}_\cdot^{t,x;\bar{u}},\bar{Z}_\cdot^{t,x;\bar{u}}) \in \mathcal{S}(t,T;\mathbb{R}^n) \times \mathcal{S}(t,T;\mathbb{R}) \times L^2_{\mathcal{F}}(t,T;\mathbb{R}^{1 \times d})$ is the corresponding optimal trajectory. If $V(\cdot,\cdot) \in C^{1,2}([0,T] \times \mathbb{R}^n)$, then
\begin{equation}\label{relation time variable}
\begin{aligned}
	-\frac{\partial V}{\partial s}(s,\bar{X}_s^{t,x;\bar{u}}) &= G\left(s,\bar{X}_s^{t,x;\bar{u}},\bar{u}_s,\frac{\partial V}{\partial x}(s,\bar{X}_s^{t,x;\bar{u}}),\frac{\partial^2 V}{\partial x^2}(s,\bar{X}_s^{t,x;\bar{u}})\right) \\
	&= \inf\limits_{u \in U} G\left(s,\bar{X}_s^{t,x;\bar{u}},u,\frac{\partial V}{\partial x}(s,\bar{X}_{s}^{t,x;\bar{u}}),\frac{\partial^2 V}{\partial x^2}(s,\bar{X}_{s}^{t,x;\bar{u}})\right),		
\end{aligned}
\end{equation}
$ \text{a.e.} s \in [t,T],\, \mathbf{P}\text{-a.s.}$, where $G$ is defined as (2.10). Moreover, if $V(\cdot,\cdot) \in C^{1,3}([0,T] \times \mathbb{R}^n)$ and $\frac{\partial^2 V}{\partial s\partial x}(\cdot,\cdot)$ is continuous, then
\begin{equation}\label{relation state variable}
\left\{
\begin{aligned}
	p_s &= \frac{\partial V}{\partial x}(s, \bar{X}_s^{t,x;\bar{u}}), && \forall s \in [t,T],\ \mathbf{P}\text{-a.s.}, \\
	q_s &= \frac{\partial^2 V}{\partial x^2}(s, \bar{X}_s^{t,x;\bar{u}}) \sigma(s, \bar{X}_s^{t,x;\bar{u}}, \bar{u}_{s}), && \text{a.e.}\ s \in [t,T],\ \mathbf{P}\text{-a.s.},
\end{aligned}
\right.
\end{equation}
where $(p_\cdot,q_\cdot)$ satisfy \eqref{first-order adjoint equation}. Furthermore, if $V(\cdot,\cdot) \in C^{1,4}([0,T] \times \mathbb{R}^n)$ and $\frac{\partial^3 V}{\partial s\partial x^2}(\cdot,\cdot)$ is continuous, then
\begin{equation}\label{relation inequality}
	\frac{\partial^2 V}{\partial x^2}(s,\bar{X}_{s}^{t,x;\bar{u}}) \leq P_s,\quad \forall s \in [t,T],\quad \mathbf{P}\text{-a.s.},
\end{equation}
where $(P_\cdot,Q_\cdot)$ satisfy \eqref{second-order adjoint equation}.
\end{mythm}

\begin{proof}
By the generalized risk-sensitive DPP (see Theorem 1 of \cite{M21}), it is easy to obtain that
\begin{equation}\label{generalized risk-sensitive DPP of Moon}
	V(s,\bar{X}_s^{t,x;\bar{u}}) = \bar{Y}_s^{t,x;\bar{u}} = \mathbb{E} \left[\int_s^T\left[\bar{l}(r)+\frac{\mu}{2}|\bar{Z}_r^{t,x;\bar{u}}|^2\right] dr + g(\bar{X}_T^{t, x; \bar{u}})\,\big|\,\mathcal{F}_s^t\right],\, \forall s \in [t,T],\, \mathbf{P}\text{-a.s.}.
\end{equation}
In fact, because
\begin{equation*}
\begin{aligned}
	V(t,x) &= J(t,x;\bar{u}_\cdot) = \bar{Y}_t^{t,x;\bar{u}}
	= \mathbb{E} \left[\int_t^T\left[\bar{l}(r)+\frac{\mu}{2}|\bar{Z}_r^{t,x;\bar{u}}|^2\right] dr +  g(\bar{X}_T^{t, x; \bar{u}})\right] \\
	&= \mathbb{E} \left[ \mathbb{E} \left[\int_t^T\left[\bar{l}(r)+\frac{\mu}{2}|\bar{Z}_r^{t,x;\bar{u}}|^2\right] dr +  g(\bar{X}_T^{t, x; \bar{u}})\,\big|\,\mathcal{F}_s^t\right]\right]	\\
	&= \mathbb{E} \left[\int_t^s\left[\bar{l}(r)+\frac{\mu}{2}|\bar{Z}_r^{t,x;\bar{u}}|^2\right] dr \right]
    + \mathbb{E} \left[ \mathbb{E} \left[\int_s^T\left[\bar{l}(r)+\frac{\mu}{2}|\bar{Z}_r^{t,x;\bar{u}}|^2\right] dr + g(\bar{X}_T^{t, x; \bar{u}})\,\big|\,\mathcal{F}_s^t\right]\right]	\\
	&= \mathbb{E} \left[\int_t^s\left[\bar{l}(r)+\frac{\mu}{2}|\bar{Z}_r^{t,x;\bar{u}}|^2\right] dr \right] + \mathbb{E}\big[J(s,\bar{X}_s^{t,x;\bar{u}};\bar{u}_\cdot)\big]\\
	&\geq \mathbb{E} \left[\int_t^s\left[\bar{l}(r)+\frac{\mu}{2}|\bar{Z}_r^{t,x;\bar{u}}|^2\right]\, dr \right] + \mathbb{E}\big[V(s,\bar{X}_s^{t,x;\bar{u}})\big]\geq V(t,x),
\end{aligned}
\end{equation*}
where the last inequality is due to the property of backward semigroup (see Theorem 1 of \cite{M21}, noting that $t \in [0,T)$ is fixed), and all the inequalities in the aforementioned become equalities. In particular,
\begin{equation*}
	\mathbb{E}\big[J(s,\bar{X}_s^{t,x;\bar{u}};\bar{u}_\cdot)\big] = \mathbb{E}[V(s,\bar{X}_s^{t,x;\bar{u}})].
\end{equation*}
However, by definition $V(s,\bar{X}_s^{t,x;\bar{u}}) \leq J(s,\bar{X}_s^{t,x;\bar{u}};\bar{u}_{\cdot}), \, \mathbf{P}\text{-a.s.}$ Thus
\begin{equation*}
	V(s,\bar{X}_s^{t,x;\bar{u}}) = J(s,\bar{X}_s^{t,x;\bar{u}};\bar{u}_\cdot), \, \mathbf{P}\text{-a.s.},
\end{equation*}
which gives \eqref{generalized risk-sensitive DPP of Moon}. For $s \in [t,T]$, define
\begin{equation*}
	m_s := \mathbb{E} \left[\int_t^T\left[\bar{l}(r)+\frac{\mu}{2}|\bar{Z}_r^{t,x;\bar{u}}|^2\right] dr + g(\bar{X}_T^{t, x; \bar{u}})\,\big|\,\mathcal{F}_s^t\right].
\end{equation*}
Clearly, $m_{\cdot}$ is a square integrable $\mathcal{F}_s^t$-martingale. Thus, by the martingale representation theorem, there exists unique $M_\cdot$ satisfying
\begin{equation*}
	m_s = m_t + \int_t^sM_r dW_r.
\end{equation*}
Thus for $s \in [t,T]$,
\begin{equation*}
	m_s = m_T - \int_s^TM_r dW_r.
\end{equation*}
Then, by \eqref{generalized risk-sensitive DPP of Moon}, we have for $s \in [t,T]$,
\begin{equation}\label{martingale representation theorem}
\begin{aligned}
	V(s,\bar{X}_s^{t,x;\bar{u}})
	&= m_s - \int_t^s\left[\bar{l}(r)+\frac{\mu}{2}|\bar{Z}_r^{t,x;\bar{u}}|^2\right] dr\\
	&= m_T - \int_s^TM_r dW_r - \int_t^s\left[\bar{l}(r)+\frac{\mu}{2}|\bar{Z}_r^{t,x;\bar{u}}|^2\right] dr\\
	&= \int_s^T\left[\bar{l}(r)+\frac{\mu}{2}|\bar{Z}_r^{t,x;\bar{u}}|^2\right] dr + V(T,\bar{X}_T^{t,x;\bar{u}}) - \int_s^TM_r dW_r.
\end{aligned}	
\end{equation}
Then, using It\^o's formula to $V(\cdot,\bar{X}_\cdot^{t,x;\bar{u}})$, we have
\begin{equation}\label{Ito formula}
\begin{aligned}
	dV(s, \bar{X}_s^{t,x;\bar{u}})
	&=\bigg\{ \frac{\partial V}{\partial s}(s, \bar{X}_s^{t,x;\bar{u}})+ \left\langle \frac{\partial V}{\partial x}(s, \bar{X}_s^{t,x;\bar{u}}), f(s, \bar{X}_s^{t,x;\bar{u}}, \bar{u}_s)\right\rangle \\
	&\qquad + \frac{1}{2}\operatorname{tr}\left[\sigma(s, \bar{X}_s^{t,x;\bar{u}}, \bar{u}_s)^\top\frac{\partial^2 V}{\partial x^2}(s, \bar{X}_s^{t,x;\bar{u}}) \sigma(s, \bar{X}_s^{t,x;\bar{u}}, \bar{u}_s)\right]\bigg\} ds \\
	&\quad + \frac{\partial V}{\partial x}(s, \bar{X}_s^{t,x;\bar{u}})^\top \sigma(s, \bar{X}_s^{t,x;\bar{u}}, \bar{u}_s) dW_{s}.
\end{aligned}
\end{equation}
Comparing \eqref{martingale representation theorem} with \eqref{Ito formula}, we conclude that, for any $s \in [t,T]$,
\begin{equation}\label{comparision}
\begin{cases}
\begin{aligned}
	&\frac{\partial V}{\partial s}(s, \bar{X}_s^{t,x;\bar{u}})+ \left\langle \frac{\partial V}{\partial x}(s, \bar{X}_s^{t,x;\bar{u}}), f(s, \bar{X}_s^{t,x;\bar{u}}, \bar{u}_s)\right\rangle \\
	&\ + \frac{1}{2}\operatorname{tr}\left[\sigma(s, \bar{X}_s^{t,x;\bar{u}}, \bar{u}_s)^\top\frac{\partial^2 V}{\partial x^2}(s, \bar{X}_s^{t,x;\bar{u}}) \sigma(s, \bar{X}_s^{t,x;\bar{u}}, \bar{u}_s)\right]
	= -\bar{l}(s) - \frac{\mu}{2}|\bar{Z}_s^{t,x;\bar{u}}|^2, \\
	&\frac{\partial V}{\partial x}(s,\bar{X}_s^{t,x;\bar{u}})^\top\sigma(s,\bar{X}_s^{t,x;\bar{u}},\bar{u}_s) = M_s,\quad \mathbf{P}\text{-a.s.}
\end{aligned}
\end{cases}
\end{equation}
However, by the uniqueness of solution to BSDE \eqref{scalar-valued BSDE}, we have for all $s \in [t,T]$,
\begin{equation}\label{uniqueness of BSDE solution}
\left\{
\begin{aligned}
	\bar{Y}_s^{t,x;\bar{u}}&= V(s,\bar{X}_s^{t,x;\bar{u}}), \\
	\bar{Z}_s^{t,x;\bar{u}}&= \frac{\partial V}{\partial x}(s,\bar{X}_s^{t,x;\bar{u}})^{\top}\sigma(s,\bar{X}_{s}^{t,x;\bar{u}},\bar{u}_{s}),\quad \mathbf{P}\text{-a.s.}
\end{aligned}
\right.
\end{equation}
Then, by \eqref{comparision} and \eqref{uniqueness of BSDE solution}, we get
\begin{equation*}
\begin{aligned}
	\frac{\partial V}{\partial s}(s,\bar{X}_s^{t,x;\bar{u}})= &- \left\langle \frac{\partial V}{\partial x}(s,\bar{X}_s^{t,x;\bar{u}}), f(s,\bar{X}_s^{t,x;\bar{u}},\bar{u}_s) \right\rangle -\bar{l}(s)
    - \frac{\mu}{2}\left|\frac{\partial V}{\partial x}(s,\bar{X}_s^{t,x;\bar{u}})^\top\sigma(s,\bar{X}_s^{t,x;\bar{u}},\bar{u}_s)\right|^2 \\
	&- \frac{1}{2} \operatorname{tr}\left[\sigma(s, \bar{X}_s^{t,x;\bar{u}}, \bar{u}_s)^\top\frac{\partial^2 V}{\partial x^2}(s, \bar{X}_s^{t,x;\bar{u}}) \sigma(s, \bar{X}_s^{t,x;\bar{u}}, \bar{u}_s)\right] \\
	&=-G\left(t,\bar{X}_s^{t,x;\bar{u}},\bar{u}_s,\frac{\partial V}{\partial x}(s,\bar{X}_s^{t,x;\bar{u}}),\frac{\partial^2 V}{\partial x^2}(s,\bar{X}_s^{t,x;\bar{u}})\right).
\end{aligned}
\end{equation*}
This proves the first equality in \eqref{relation time variable}. Since $V(\cdot,\cdot) \in C^{1,2}([0,T] \times \mathbb{R}^n)$, it satisfies the HJB equation \eqref{HJB equation}, which implies the second equality in \eqref{relation time variable}. Also, by \eqref{HJB equation} we have
\begin{equation}
\begin{aligned}
	0&=\frac{\partial V}{\partial s}(s,\bar{X}_s^{t,x;\bar{u}}) + G\left(s,\bar{X}_s^{t,x;\bar{u}},\bar{u}_s,\frac{\partial V}{\partial x}(s,\bar{X}_s^{t,x;\bar{u}}),\frac{\partial^2 V}{\partial x^2}(s,\bar{X}_s^{t,x;\bar{u}})\right)\\
	&\leq \frac{\partial V}{\partial s}(s,x) + G\left(s,x,\bar{u}_{s},\frac{\partial V}{\partial x}(s,x),\frac{\partial^2 V}{\partial x^2}(s,x)\right).
\end{aligned}
\end{equation}
Consequently, if $V(\cdot,\cdot) \in C^{1,3}([0,T] \times \mathbb{R}^n)$ with $\frac{\partial^2 V}{\partial s\partial x}$ being also continuous, then
\begin{equation}\label{first order min condition}
    \left.\frac{\partial}{\partial x}
	\middle\{
	\frac{\partial V}{\partial s}(s,x) + G\left(s,x,\bar{u}_{s},\frac{\partial V}{\partial x}(s,x),\frac{\partial^2 V}{\partial x^2}(s,x)\right)
	\middle\}
	\right|_{x=\bar{X}_s^{t,x;\bar{u}}} = 0,\quad \forall s \in [t,T],
\end{equation}
which is the first-order minimum condition. Furthermore, if  $V(\cdot,\cdot) \in C^{1,4}([0,T] \times \mathbb{R}^n)$ and $\frac{\partial^3 V}{\partial s\partial x^2}$ is continuous, then the second-order minimum condition holds:
\begin{equation}\label{second order min condition}
	\left.\frac{\partial^2}{\partial x^2}
	\middle\{
	\frac{\partial V}{\partial s}(s,x) + G\left(s,x,\bar{u}_{s},\frac{\partial V}{\partial x}(s,x),\frac{\partial^2 V}{\partial x^2}(s,x)\right)
	\middle\}
	\right|_{x=\bar{X}_s^{t,x;\bar{u}}} \geq 0,\quad \forall s \in [t,T].
\end{equation}
On the one hand, \eqref{first order min condition} yields that (recall \eqref{generalized Hamiltonian}), for any $s \in [t,T]$,
\begin{equation}\label{first order min condition 1}
\begin{aligned}
	0=&\ \frac{\partial^2 V}{\partial s\partial x}(s,\bar{X}_s^{t,x;\bar{u}}) + \frac{\partial^2 V}{\partial x^2}(s,\bar{X}_s^{t,x;\bar{u}})\bar{f}(s)
    + \frac{\partial \bar f}{\partial x}(s)^\top \frac{\partial V}{\partial x}(s,\bar{X}_s^{t,x;\bar{u}})\\
	&+\frac{\partial \bar l}{\partial x}(s) + \frac{1}{2} \operatorname{tr}\left[\bar{\sigma}(s)^\top\frac{\partial^3 V}{\partial x^3}(s,\bar{X}_s^{t,x;\bar{u}})\bar{\sigma}(s)\right]
    + \sum_{j=1}^d \frac{\partial \bar \sigma_j}{\partial x}(s)^\top \frac{\partial^2 V}{\partial x^2}(s,\bar{X}_s^{t,x;\bar{u}})\bar{\sigma}_j(s)\\
	&+\mu\sum_{j=1}^d \bar{\sigma}_j(s)^\top \frac{\partial V}{\partial x}(s,\bar{X}_s^{t,x;\bar{u}})\frac{\partial \bar \sigma_j}{\partial x}(s)^\top \frac{\partial V}{\partial x}(s,\bar{X}_s^{t,x;\bar{u}}) \\
	&+ \mu\sum_{j=1}^d \frac{\partial^2 V}{\partial x^2}(s,\bar{X}_s^{t,x;\bar{u}})\bar{\sigma}_j(s)\bar{\sigma}_j(s)^\top\frac{\partial V}{\partial x}(s,\bar{X}_s^{t,x;\bar{u}}),
\end{aligned}
\end{equation}
where $\operatorname{tr}\left(\bar{\sigma}^\top \frac{\partial^3 V}{\partial x^3}\bar{\sigma}\right):= \left(\operatorname{tr}\left(\bar{\sigma}^\top\frac{\partial^2 (\frac{\partial V}{\partial x})^1}{\partial x^2}\bar{\sigma}\right),\cdots,\operatorname{tr}\left(\bar{\sigma}^\top\frac{\partial^2 (\frac{\partial V}{\partial x})^n}{\partial x^2}\bar{\sigma}\right)\right)^{\top}$, with $\left((\frac{\partial V}{\partial x})^1,\cdots,(\frac{\partial V}{\partial x})^n\right)^\top = \frac{\partial V}{\partial x}$. On the other hand, applying It\^o's formula to $\frac{\partial V}{\partial x}(\cdot,\bar{X}_\cdot^{t,x;\bar{u}})$, we get
\begin{equation}\label{Ito formula 1}
\begin{aligned}
	& d\frac{\partial V}{\partial x}(s, \bar{X}_{s}^{t,x;\bar{u}})
    = \bigg\{ \frac{\partial^2 V}{\partial s\partial x}(s, \bar{X}_s^{t,x;\bar{u}})+ \frac{\partial^2 V}{\partial x^2}(s, \bar{X}_s^{t,x;\bar{u}}) \bar{f}(s)\\
	&\quad + \frac{1}{2} \operatorname{tr}\left[ \bar{\sigma}(s)^\top \frac{\partial^3 V}{\partial x^3}(s, \bar{X}_s^{t,x;\bar{u}}) \bar{\sigma}(s) \right] \bigg\} ds
    + \frac{\partial^2 V}{\partial x^2}(s, \bar{X}_s^{t,x;\bar{u}}) \bar{\sigma}(s) dW_s \\
	&= -\bigg\{ \frac{\partial \bar f}{\partial x}(s)^\top \frac{\partial V}{\partial x}(s, \bar{X}_s^{t,x;\bar{u}})+ \frac{\partial \bar l}{\partial x}(s)
	+ \sum_{j=1}^d \frac{\partial \bar \sigma_j}{\partial x}(s)^\top \frac{\partial^2 V}{\partial x^2}(s,\bar{X}_s^{t,x;\bar{u}})\bar{\sigma}_j(s)  \\
	&\quad + \mu \sum_{j=1}^d \bar{\sigma}_j(s)^\top \frac{\partial V}{\partial x}(s, \bar{X}_{s}^{t,x;\bar{u}})\frac{\partial \bar \sigma_j}{\partial x}(s)^\top \frac{\partial V}{\partial x}(s, \bar{X}_s^{t,x;\bar{u}}) \\
	&\quad + \mu \sum_{j=1}^d \frac{\partial^2 V}{\partial x^2}(s,\bar{X}_s^{t,x;\bar{u}})\bar{\sigma}_j(s) \bar{\sigma}_j(s)^\top \frac{\partial V}{\partial x}(s, \bar{X}_{s}^{t,x;\bar{u}}) \bigg\} ds
	+ \frac{\partial^2 V}{\partial x^2}(s, \bar{X}_s^{t,x;\bar{u}}) \bar{\sigma}(s) dW_s.
\end{aligned}
\end{equation}
Note that $V(\cdot,\cdot)$ solves \eqref{HJB equation}, and thus $\frac{\partial V}{\partial x}(T, \bar{X}_T^{t,x;\bar{u}}) = \frac{\partial g}{\partial x}(\bar{X}_T^{t,x;\bar{u}})$. Hence, by the uniqueness of the solutions to \eqref{first-order adjoint equation}, we obtain \eqref{relation state variable}.

Moreover, \eqref{Ito formula 1} yields that, for $s \in [t,T]$,
\begin{equation}\label{Ito formula 2}
\begin{aligned}
	&\frac{\partial^3 V}{\partial s\partial x^2}(s,\bar{X}_s^{t,x;\bar{u}}) + \frac{\partial^3 V}{\partial x^3}(s,\bar{X}_s^{t,x;\bar{u}})\bar{f}(s)
    +\frac{\partial^2 V}{\partial x^2}(s,\bar{X}_{s}^{t,x;\bar{u}})\frac{\partial \bar f}{\partial x}(s) +\frac{\partial^2 l}{\partial x^2}(s)\\
    &+\frac{\partial^2 f}{\partial x^2}(s)^\top \frac{\partial V}{\partial x}(s,\bar{X}_{s}^{t,x;\bar{u}})+ \frac{\partial \bar f}{\partial x}(s)^\top \frac{\partial^2 V}{\partial x^2}(s,\bar{X}_s^{t,x;\bar{u}})\\
	&+ \frac{1}{2} \operatorname{tr}\left[\bar{\sigma}(s)^\top \frac{\partial^4 V}{\partial x^4}(s,\bar{X}_s^{t,x;\bar{u}})\bar{\sigma}(s)\right]
    +\sum_{j=1}^d \frac{\partial^2 \bar \sigma_j}{\partial x^2}(s)^\top \frac{\partial^2 V}{\partial x^2}(s,\bar{X}_s^{t,x;\bar{u}})\bar{\sigma}_j(s)\\
	&+\sum_{j=1}^d \frac{\partial \bar \sigma_j}{\partial x}(s)^\top\left(\frac{\partial^2 V}{\partial x^2}(s,\bar{X}_s^{t,x;\bar{u}})
    + \mu \frac{\partial V}{\partial x}(s,\bar{X}_s^{t,x;\bar{u}})\frac{\partial V}{\partial x}(s,\bar{X}_{s}^{t,x;\bar{u}})^{\top}\right)\frac{\partial \bar \sigma_j}{\partial x}(s) \\
	&+ \sum_{j=1}^d \frac{\partial \bar \sigma_j}{\partial x}(s)^\top\left(\frac{\partial^3 V}{\partial x^3}(s, \bar{X}_s^{t,x;\bar{u}})\bar{\sigma}_j(s)
    + \mu \frac{\partial^2 V}{\partial x^2}(s,\bar{X}_{s}^{t,x;\bar{u}})\bar{\sigma}_j(s)\frac{\partial V}{\partial x}(s,\bar{X}_s^{t,x;\bar{u}})^\top\right.\\
	&\quad +\left.\mu \frac{\partial V}{\partial x}(s,\bar{X}_s^{t,x;\bar{u}})\bar{\sigma}_j(s)^\top \frac{\partial^2 V}{\partial x^2}(s, \bar{X}_s^{t,x;\bar{u}})^\top\right) \\
	&+ \sum_{j=1}^d \left(\frac{\partial^3 V}{\partial x^3}(s, \bar{X}_s^{t,x;\bar{u}})\bar{\sigma}_j(s)
    + \mu \frac{\partial^2 V}{\partial x^2}(s,\bar{X}_{s}^{t,x;\bar{u}})\bar{\sigma}_j(s)\frac{\partial V}{\partial x}(s,\bar{X}_s^{t,x;\bar{u}})\right.\\
	&\quad +\left.\mu \frac{\partial^2 V}{\partial x^2}(s, \bar{X}_s^{t,x;\bar{u}})\bar{\sigma}_j(s)\frac{\partial V}{\partial x}(s, \bar{X}_s^{t,x;\bar{u}})^{\top}\right)\frac{\partial \bar \sigma_j}{\partial x}(s) \\
	&+ \mu\sum_{j=1}^d \frac{\partial^2 V}{\partial x^2}(s, \bar{X}_s^{t,x;\bar{u}})\bar{\sigma}_j(s)\bar{\sigma}_j(s)^\top \frac{\partial^2 V}{\partial x^2}(s, \bar{X}_s^{t,x;\bar{u}})^\top \\
    &+ \mu\sum_{j=1}^d \frac{\partial V}{\partial x}(s, \bar{X}_s^{t,x;\bar{u}})\frac{\partial^2 \bar \sigma_j}{\partial x^2}(s)\bar{\sigma}_j(s)^\top \frac{\partial V}{\partial x}(s, \bar{X}_s^{t,x;\bar{u}})^\top\\
	&+ \mu\sum_{j=1}^d \frac{\partial^3 V}{\partial x^3}(s, \bar{X}_s^{t,x;\bar{u}})\bar{\sigma}_j(s)\bar{\sigma}_j(s)^\top \frac{\partial V}{\partial x}(s,\bar{X}_s^{t,x;\bar{u}})\geq0.
\end{aligned}
\end{equation}
In the above and what follows, the notation of partial derivatives has its own definitions which we will not clarify on by one, because of limited space. (For simplicity, we can verify the calculus
just using $n=1$, i.e., $x$ is one-dimensional.)

Applying It\^o's formula to $\frac{\partial^2 V}{\partial x^2}(\cdot,\bar{X}_\cdot^{t,x;\bar{u}})$, we get
\begin{equation}\label{Ito formula 3}
\begin{aligned}
	& d\frac{\partial^2 V}{\partial x^2}(s, \bar{X}_s^{t,x;\bar{u}}) = \bigg\{ \frac{\partial^3 V}{\partial s\partial x^2}(s, \bar{X}_s^{t,x;\bar{u}})
	 + \frac{\partial^3 V}{\partial x^3}(s, \bar{X}_s^{t,x;\bar{u}}) \bar{f}(s)\\
	&\quad + \frac{1}{2} \operatorname{tr}\left[ \bar{\sigma}(s)^\top \frac{\partial^4 V}{\partial x^4}(s, \bar{X}_s^{t,x;\bar{u}}) \bar{\sigma}(s) \right] \bigg\} ds
	 + \sum_{j=1}^d \frac{\partial^3 V}{\partial x^3}(s, \bar{X}_s^{t,x;\bar{u}}) \bar{\sigma}_j(s) dW_{js}.
\end{aligned}
\end{equation}
For all $s \in [t,T]$, define
\begin{equation}\label{tilde P Q}
	\tilde{P}_s := \frac{\partial^2 V}{\partial x^2}(s, \bar{X}_s^{t,x;\bar{u}}), \quad \tilde{Q}_{js} := \frac{\partial^3 V}{\partial x^3}(s, \bar{X}_s^{t,x;\bar{u}})\bar{\sigma}_j(s),
\end{equation}
and we have
\begin{equation}\label{BSDE of tilde P Q}
\begin{aligned}
	-d\tilde{P}_s = -&\bigg\{ \frac{\partial^3 V}{\partial s\partial x^2}(s, \bar{X}_s^{t,x;\bar{u}})+ \frac{\partial^3 V}{\partial x^3}(s, \bar{X}_s^{t,x;\bar{u}}) \bar{f}(s) \\
	&\quad + \frac{1}{2} \operatorname{tr}\left[ \bar{\sigma}(s)^\top \frac{\partial^4 V}{\partial x^4}(s, \bar{X}_s^{t,x;\bar{u}}) \bar{\sigma}(s) \right] \bigg\} ds  - \sum_{j=1}^d\tilde{Q}_{js}dW_{js}.
\end{aligned}
\end{equation}
From \eqref{tilde P Q} and the continuity of $\frac{\partial^3 V}{\partial s\partial x^2}(\cdot, \cdot)$ as well as \eqref{relation state variable}, we have
\begin{equation}
\begin{aligned}
	&-\bigg\{ \frac{\partial^3 V}{\partial s\partial x^2}(s, \bar{X}_s^{t,x;\bar{u}})+ \frac{\partial^3 V}{\partial x^3}(s, \bar{X}_s^{t,x;\bar{u}}) \bar{f}(s)
    + \frac{1}{2} \operatorname{tr}\left[ \bar{\sigma}(s)^\top \frac{\partial^4 V}{\partial x^4}(s, \bar{X}_s^{t,x;\bar{u}}) \bar{\sigma}(s) \right] \bigg\} \\
	\leq\ & \frac{\partial \bar f}{\partial x}(s)^\top\tilde{P}_s +\tilde{P}_s\frac{\partial \bar f}{\partial x}(s)
    + \sum_{j=1}^d \frac{\partial \bar \sigma_j}{\partial x}(s)^\top\left(\tilde{P}_s + \mu p_sp_s^\top\right)\frac{\partial \bar \sigma_j}{\partial x}(s)+ \mu\sum_{j=1}^d q_{js}q_{js}^\top \\
	&+ \sum_{j=1}^d \frac{\partial \bar \sigma_j}{\partial x}(s)^\top\left(\tilde{Q}_{js} + \mu \tilde{P}_s\bar{\sigma}_j(s)p_s^\top +\mu p_sq_{js}^\top\right) \\
	&+ \sum_{j=1}^d \left(\tilde{Q}_{js} + \mu \tilde{P}_s\bar{\sigma}_j(s)p_s^\top+\mu q_{js}p_s^\top\right)\frac{\partial \bar \sigma_j}{\partial x}(s)
	+ \frac{\partial^2 \bar H}{\partial x^2}(s)+ \mu\sum_{j=1}^d \tilde{Q}_{js}\bar{\sigma}_j(s)^\top p_s.
\end{aligned}
\end{equation}
Note that $P_T = \tilde{P}_T = \frac{\partial^2 V}{\partial x^2}(T, \bar{X}_T^{t,x;\bar{u}}) = \frac{\partial^2 g}{\partial x^2}( \bar{X}_T^{t,x;\bar{u}})$. Using the above relation, one can check that the condition  of the comparison theorem (see, Hu and Peng \cite{HP06}) for the matrix-valued BSDEs \eqref{second-order adjoint equation} and \eqref{BSDE of tilde P Q} holds, and thus we obtain \eqref{relation inequality}. The proof is complete.
\end{proof}

\section{Applications to financial portfolio optimization}

In this section, we consider a risk-sensitive portfolio optimization problem for a factor model (see \cite{BP99,KN02}). In this problem, the optimal portfolio in the state feedback form is obtained by both MP and DPP approaches, and the relations in Theorem \ref{relation of MP and DPP} are illustrated explicitly.

We make a slight modification to the factor model studied by Kuroda and Nagai \cite{KN02}, considering a market with two kinds of securities and a factor.

(i) A risk-free security (e.g., a bond), where the price $R_t$ at time $t$ is given by
\begin{equation}\label{bond}
	dR_t = r_tR_t\,dt, \quad R_0>0.
\end{equation}

(ii) A risky security (e.g., a stock), where the price $S_t$ at time $t$ is given by
\begin{equation}\label{stock}
	dS_t = S_t\left\{(a + AX_t)\,dt + \sigma d\,W_t\right\}, \quad S_0>0.
\end{equation}

(iii) A factor (e.g., short-term interest rates), where the level $X_t$ at time $t$ is given by
\begin{equation}\label{factor}
	dX_t = (b + BX_t)\,dt + \Lambda d\,W_t, \quad X_0=x,
\end{equation}
where $W_t\equiv(W_t^1,W_t^2)^\top$ is an $\mathbb{R}^2$-valued standard Brownian motion defined on a filtered probability space $(\Omega, \mathcal{F}, \{\mathcal{F}_t\}_{t\geq 0}, \mathbf{P})$. Here $a,A,b,B \in \mathbb{R}$ are constants, $r_\cdot$ is a bounded deterministic function. $\sigma = [\sigma_1,\sigma_2] \in \mathbb{R}^{1\times2}$ and $\Lambda = [\Lambda_1,\Lambda_2] \in \mathbb{R}^{1\times2}$ are constant vectors with $\sigma\sigma^\top > 0$.

At time $t$, let $u_t$ denote the investment strategy of the investor in the risky security, then $1-u_t$ denote that in the risk-free security. Set $\mathcal{G}_{t} := \sigma(S_r,X_r;r \leq t)$, and we define
\begin{equation}
\mathcal{U}[0,T] := \left\{u \colon [0,T] \times \Omega \to \mathbb{R} \;\Big|\; u \text{ is } \mathcal{G}_t\text{-progressively measurable},\ \mathbb{E} \left[ \int_0^T |u(t)|^2\, dt \right] < \infty \right\},
\end{equation}
which denote the set of all admissible investment strategies. For given $u_\cdot \in \mathcal U[0,T]$, $V_t$ representing the investor's wealth at time $t$ is given as follows:
\begin{equation}\label{wealth}
\left\{
 \begin{aligned}
 	dV_t&= V_t\big[r_t\,dt + u_t(a + AX_t-r_t)\,dt + u_t\sigma\,dW_t\big],\\
 	V_0&= v.
 \end{aligned}
\right.
\end{equation}

For given initial wealth $v > 0$ and $\theta > 0$, the investor wants to choose an optimal portfolio $\bar{u}_{\cdot} \in \mathcal U[0,T]$ to maximize the following risk-sensitive expected growth rate up to time horizon $T$:
\begin{equation}\label{bjective functional}
	\tilde{J}(u_\cdot) = -\frac{1}{\theta}\log \mathbb{E} \big[\exp\left\{-\theta \log(V_T)\right\}\big].
\end{equation}

Noting that $V_\cdot$ satisfies \eqref{wealth}, we have
\begin{equation}
	\begin{aligned}
		\exp\left\{-\theta \log(V_t)\right\}=v^{-\theta}\exp\left\{\theta \int_0^t l(s,X_s,u_s)\,ds - \frac{1}{2}\theta^2 \int_0^t\sigma\sigma^\top u_s^2\,ds - \theta \int_0^tu_{s}\sigma\,dW_s\right\},
	\end{aligned}
\end{equation}
where $l(s,X,u) := \frac{1}{2}(\theta + 1)\sigma\sigma^\top u^2 - r_s -u(a + AX- r_s)$.

We introduce a new probability measure $\hat{\mathbf{P}}$ on $(\Omega, \mathcal{F})$ defined by
\begin{equation*}
	\left. \frac{d\hat{\mathbf{P}}}{d\mathbf{P}}\right|_{\mathcal{F}_t} = Z_t,
\end{equation*}
where $Z_t := \exp\left\{- \frac{1}{2}\theta^2 \int_0^t\sigma\sigma^\top u_s^2\,ds - \theta \int_0^tu_s\sigma\,dW_s\right\}$.

We denote by $\mathcal{V}[0,T]$ the set of investment strategies $u_\cdot \in \mathcal{U}[0,T]$  such that $\hat{\mathbf{P}}$ is a probability measure, i.e., $\mathbb{E}[Z_t] = 1$. According to Girsanov's
theorem, for any $u_\cdot \in \mathcal{V}[0,T]$,
\begin{equation*}
	\hat{W}_t := W_t + \theta \int_0^t\sigma^\top u_s\,ds
\end{equation*}
is a standard Brownian motion under the probability measure $\hat{\mathbf{P}}$ and
\begin{equation}\label{wealth new}
\left\{
\begin{aligned}
	dX_t &= (b + BX_t - \theta\Lambda\sigma^\top u_s)\,dt + \Lambda\,d\hat{W}_t,\\
     X_0 &= x.
\end{aligned}
\right.
\end{equation}

We can now introduce the auxiliary objective functional under the measure $\hat{\mathbf{P}}$:
\begin{equation}\label{auxiliary objective functional}
	\hat{J}(u_\cdot) = \log v -\frac{1}{\theta}\log \hat{\mathbb{E}} \left[\exp\left\{\theta \int_0^T l(s,X_s,u_s)\,ds\right\}\right],
\end{equation}
where $\hat{\mathbb{E}}[\cdot]$ denotes the expectation with respect to $\hat{\mathbf{P}}$. Let $J(u_\cdot)=-\hat{J}(u_\cdot)$, then our portfolio optimization problem can be rewritten as
\begin{equation}\label{portfolio optimization}
	J(\bar{u}_\cdot) = \inf\limits_{u_\cdot \in \mathcal{V}[0,T]}J(u_\cdot).
\end{equation}

Since we are going to involve the DPP in treating the above problem, we will adopt the formulation as in Section 2. Let $T>0$ be given. For any $(t,x) \in [0,T] \times \mathbb{R}$, by \eqref{wealth new},
consider the following SDE:
\begin{equation}\label{wealth new 1}
\left\{
\begin{aligned}
	dX_s^{t,x;u} &= \big(b + BX_s^{t,x;u} - \theta\Lambda\sigma^\top u_s)\,ds + \Lambda\,d\hat{W}_s,\\
     X_0^{t,x;u} &= x.
\end{aligned}
\right.
\end{equation}
And our risk-sensitive portfolio optimization problem is to find an optimal portfolio $\bar{u}_\cdot \in \mathcal{V}[0,T]$ to minimize the cost functional $J(t,x;u_\cdot) = -\log v +\frac{1}{\theta}\log \hat{\mathbb{E}} \left[\exp\left\{\theta \int_t^T l(s,X_s^{t,x;u},u_s)\,ds\right\}\right]$. We define the value function as
\begin{equation}\label{value function example}
	V(t,x) := J(t,x;\bar{u}_\cdot) = \inf\limits_{u_\cdot \in \mathcal{V}[t,T]}J(t,x;u_\cdot).
\end{equation}

We can check that all the assumptions in Section 2 are satisfied. Then, we can use both DPP and MP approached to solve problem \eqref{value function example}.

\subsection{MP Approach}

In this case, the Hamiltonian function given by
\begin{equation}\label{Hamiltonian example}
\begin{aligned}
 	H(s,x,u,p,q) &= p\big(b + Bx-\theta\Lambda\sigma^\top u\big) + \frac{1}{2}(\theta + 1)\sigma\sigma^\top u^2 - r_s \\
 	&\quad -u(a + AX_s- r_s)+ q^\top\Lambda + \theta\Lambda_1^2p^2 + \theta\Lambda_2^2p^2,
\end{aligned}
\end{equation}
and the corresponding $\mathcal H$-function is
\begin{equation}\label{mathcal Hamiltonian example}
	\mathcal{H}(s,x,u) = H(s,x,u,p,q) - \frac{1}{2} \Lambda_1^2(P + \theta p^2)- \frac{1}{2} \Lambda_2^2(P + \theta p^2).
\end{equation}
The first-and second-order adjoint equations are
\begin{equation}\label{first-order adjoint equation example}
\left\{
\begin{aligned}
	dp_s&=-\big(Bp_s - Au_s + \theta \Lambda q_s^\top p_s\big)\,ds + q_s\,d\hat{W}_s,\\
	p_T&=0,
\end{aligned}
\right.
\end{equation}
and
\begin{equation}\label{second-order adjoint equation example}
\left\{
\begin{aligned}
	dP_s&=-\big(2BP_s + \theta q_sq_s^\top + \theta Q_s \Lambda^\top p_s\big)\,ds + Q_s\,d\hat{W}_s,\\
	P_T&=0,
\end{aligned}
\right.
\end{equation}
respectively. Let $\bar{u}_\cdot$ be a candidate optimal portfolio, $\bar{X}_\cdot^{t,x,\bar{u}}$ be the corresponding solution to the controlled SDE \eqref{wealth new 1}, and $(\bar{p}_\cdot,\bar{q}_\cdot)$ and $(\bar{P}_\cdot,\bar{Q}_\cdot)$ be the corresponding solution to the adjoint equations \eqref{first-order adjoint equation example} and \eqref{second-order adjoint equation example}, respectively. Now the $\mathcal{H}$-function \eqref{mathcal Hamiltonian example} is
\begin{equation}\label{mathcal Hamiltonian example 1}\hspace{-2mm}
\begin{aligned}
	\mathcal{H}(s,\bar{X}_s^{t,x;\bar{u}},u) &=
	\bar{p}_s\big(b + B\bar{X}_s^{t,x;\bar{u}}-\theta\Lambda\sigma^\top u\big) + \frac{1}{2}(\theta + 1)\sigma\sigma^\top u^2 - r_s -u\big(a + A\bar{X}_s^{t,x;\bar{u}}- r_s\big)\\
	&\quad + \bar{q}_s^\top\Lambda + \theta\Lambda_1^2\bar{p}_s^2 + \theta\Lambda_2^2\bar{p}_s^2- \frac{1}{2} \Lambda_1^2(\bar{P}_s + \theta \bar{p}_s^2)- \frac{1}{2} \Lambda_2^2(\bar{P}_s + \theta \bar{p}_s^2).
\end{aligned}
\end{equation}
Then, by the minimum condition \eqref{maximum condition-mathcal H}, we have
\begin{equation}\label{optimal portfolio}
	\bar{u}_s=\frac{1}{\theta+1}(\sigma\sigma^\top)^{-1}\left(\theta\bar{p}_s\Lambda\sigma^\top+a + A\bar{X}_s^{t,x;\bar{u}}- r_s\right),\quad \text{a.e.} s \in [t,T],\quad \hat{\mathbf{P}}\text{-a.s.}
\end{equation}
Therefore, an adapted solution triple $(\bar{X}_\cdot^{t,x;\bar{u}},\bar{p}_\cdot,\bar{q}_\cdot)$ can be obtained by solving the following FBSDE:
\begin{equation}\label{FBSDE example}
\left\{
\begin{aligned}
	d\bar{X}_s^{t,x;\bar{u}}&= \big(b + B\bar{X}_s^{t,x;\bar{u}} - \theta\Lambda\sigma^\top\bar{u}_s\big)\,ds + \Lambda\,d\hat{W}_s,\\
	d\bar{p}_s&= -\big(B\bar{p}_s - A\bar{u}_s + \theta \Lambda \bar{q}_s^\top\bar{p}_s\big)\,ds + \bar{q}_s\,d\hat{W}_s,\\
	\bar{X}_t^{t,x;\bar{u}}&= x,\quad \bar{p}_T=0.
\end{aligned}
\right.
\end{equation}
To solve \eqref{FBSDE example}, we conjecture the solution is related by the following:
\begin{equation}\label{relation p and X}
	\bar{p}_s=\Gamma_s\bar{X}_s^{t,x;\bar{u}} + \phi_s,\quad s \in [t,T],
\end{equation}
where $\Gamma_\cdot \in C^1([t,T];\mathbb{R})$ and $\phi_\cdot \in C^1([t,T];\mathbb{R})$, with $\Gamma_T= 0,\,\phi_T=0$. By applying It\^o's formula to \eqref{relation p and X}, we get
\begin{equation}\label{Ito formula example}
\begin{aligned}
 	d\bar{p}_s=\left\{\dot{\Gamma_{s}}\bar{X}_s^{t,x;\bar{u}} + \Gamma_{s}\big(b + B\bar{X}_s^{t,x;\bar{u}} - \theta\Lambda\sigma^\top\bar{u}_s\big) + \dot{\phi_s} \right \}ds + \Gamma_s\Lambda \,d\hat{W}_s.
\end{aligned}
\end{equation}
On the other hand, after substituting \eqref{relation p and X} into the backward equation of \eqref{FBSDE example}, we arrive at
\begin{equation}\label{substituting}
 	d\bar{p}_s=-\big(B\Gamma_s\bar{X}_s^{t,x;\bar{u}} + B\phi_s - A\bar{u}_s + \theta \Lambda \bar{q}_s^\top\Gamma_s\bar{X}_s^{t,x;\bar{u}} + \theta \Lambda \bar{q}_s^\top\phi_s\big)\,ds + \bar{q}_s\,d\hat{W}_s.
\end{equation}
Equating the coefficients of \eqref{Ito formula example} and \eqref{substituting}, gives
\begin{equation}\label{relarion p q and X example}
	(\bar{p}_s,\bar{q}_s) = \big(\Gamma_s\bar{X}_s^{t,x;\bar{u}} + \phi_s,\Gamma_s\Lambda\big),\quad \text{a.e.} s \in [t,T],\quad \hat{\mathbf{P}}\text{-a.s.},
\end{equation}
where $\Gamma_\cdot$ is the solution to the following Riccati type equation:
\begin{equation}\label{Riccati}
\left\{
\begin{aligned}
	& \dot{\Gamma}_s
	+ \left(\theta\Lambda\Lambda^\top - \frac{\theta^2}{\theta +1}\Lambda\sigma^\top(\sigma\sigma^\top)^{-1}\sigma\Lambda^\top\right)\Gamma_s^2 \\
	&\quad + 2\left(B - \frac{\theta}{\theta +1}A\Lambda\sigma^\top(\sigma\sigma^\top)^{-1}\right)\Gamma_s- \frac{1}{\theta +1}A^2(\sigma\sigma^\top)^{-1} = 0, \\
	& \Gamma_T = 0,
\end{aligned}
\right.
\end{equation}
and $\phi_\cdot$ is the solution to the following equation
\begin{equation}\label{ODE}
\left\{
\begin{aligned}
	&\dot{\phi}_s + \Biggl[B - \frac{\theta}{\theta +1}A\Lambda\sigma^\top(\sigma\sigma^\top)^{-1}
	+ \left(\theta\Lambda\Lambda^\top - \frac{\theta^2}{\theta +1}\Lambda\sigma^\top(\sigma\sigma^\top)^{-1}\sigma\Lambda^\top\right)\Gamma_s\Biggr]\phi_s \\
	&\quad + b\Gamma_s - \left[\frac{1}{\theta +1}A(\sigma\sigma^\top)^{-1} + \frac{\theta}{\theta +1}\Lambda\sigma^\top(\sigma\sigma^\top)^{-1}\Gamma_s\right](a - r_s) = 0, \\
	&\phi_T = 0.
\end{aligned}
\right.
\end{equation}
Finally, by \eqref{optimal portfolio} and \eqref{relation p and X}, we can get the optimal portfolio in the following state feedback form:
\begin{equation}\label{optimal portfolio state feedback form}\hspace{-1mm}
	\bar{u}_s=\frac{1}{\theta+1}(\sigma\sigma^\top)^{-1} \left[ \big(\theta\Lambda\sigma^\top\Gamma_{s} + A\big)\bar{X}_s^{t,x;\bar{u}}
	+ \theta\Lambda\sigma^\top\phi_s + a - r_s\right],\quad \text{a.e.} s \in [t,T],\quad \hat{\mathbf{P}}\text{-a.s.},
\end{equation}
where $\Gamma_\cdot, \phi_\cdot$ are given by \eqref{Riccati} and \eqref{ODE}, respectively.

To summarize, we have the following result.
\begin{mythm}\label{MP example}
If equations \eqref{Riccati} and \eqref{ODE} admit solutions $\Gamma_\cdot$ and $\phi_\cdot$, then the optimal control $\bar{u}_{\cdot}$ of our risk-sensitive portfolio optimization problem \eqref{wealth new 1}-\eqref{value function example} has the following state feedback form
\begin{equation}
	\bar{u}(s,\bar{X})=\frac{1}{\theta+1}(\sigma\sigma^\top)^{-1} \left[ \big(\theta\Lambda\sigma^\top\Gamma_{s} + A\big)\bar{X}+ \theta\Lambda\sigma^\top\phi_s + a - r_s\right],
\end{equation}
for $\text{a.e.} s \in [0,T]$, $\hat{\mathbf{P}}\text{-a.s.}$, where $\Gamma_\cdot, \phi_\cdot$ are given by \eqref{Riccati} and \eqref{ODE}, respectively.
\end{mythm}

\subsection{DPP approach}

In this case, the value function $V(\cdot,\cdot)$ should satisfy the following HJB equation:
\begin{equation}\label{HJB equation example}
\left\{
\begin{aligned}
	&-\frac{\partial V}{\partial t}(t,x) = \inf\limits_{u \in U} G\left(t,x,u,\frac{\partial V}{\partial x}(t,x),\frac{\partial^2 V}{\partial x^2}(t,x)\right),\\
    &V(T,x)=-\log v,		
\end{aligned}
\right.
\end{equation}
where the generalized Hamiltonian function \eqref{generalized Hamiltonian} is
\begin{equation}\label{generalized Hamiltonian example}
\begin{aligned}
	&G\left(t,x,u,\frac{\partial V}{\partial x}(t,x),\frac{\partial^2 V}{\partial x^2}(t,x)\right)\\
    &= \frac{1}{2}(\theta + 1)\sigma\sigma^\top u^2 - r_t - (a + Ax-r_t)u + \frac{\partial V}{\partial x}(t,x)\big(b+Bx-\theta\Lambda\sigma^\top u\big)  \\
	&\quad+ \frac{\theta}{2}\Lambda\Lambda^\top \frac{\partial V}{\partial x}(t,x)^2+ \frac{1}{2}\Lambda\Lambda^\top \frac{\partial^2 V}{\partial x^2}(t,x).
\end{aligned}	
\end{equation}
We conjecture that $V(t,x)$ is quadratic in $x$, namely,
\begin{equation}\label{conjecture of V(t,x)}
	V(t,x) = \frac{1}{2}\Psi_tx^2 + \eta_tx +k_t,
\end{equation}
for some deterministic differentiable functions $\Psi_\cdot,\eta_\cdot$ and $k_\cdot$ with $\Psi_T= 0,\,\eta_T=0,\,k_T=-\log v$.

Substituting \eqref{conjecture of V(t,x)} into \eqref{generalized Hamiltonian example} and using completion of squares, we get
\begin{equation}
\begin{aligned}
	&G\left(t,x,u,\frac{\partial V}{\partial x}(t,x),\frac{\partial^2 V}{\partial x^2}(t,x)\right) \\
	&= \frac{1}{2}(\theta + 1)\sigma\sigma^\top \bigg\{ u - \frac{1}{\theta+1}(\sigma\sigma^\top)^{-1} \Big[ \big(\theta\Lambda\sigma^\top\Psi_t + A\big)x+ \theta\Lambda\sigma^\top\eta_t + a - r_t \Big] \bigg\}^2 \\
	&\quad + \Bigg\{\frac{1}{2}\left[\theta\Lambda\Lambda^\top - \frac{\theta^2}{\theta +1}\Lambda\sigma^\top(\sigma\sigma^\top)^{-1}\sigma\Lambda^\top\right]\Psi_t^2
	+ \left(B - \frac{\theta}{\theta +1}A\Lambda\sigma^\top(\sigma\sigma^\top)^{-1}\right)\Psi_t \\
	&\qquad - \frac{1}{2(\theta +1)}A^2(\sigma\sigma^\top)^{-1}\Bigg\}x^2 \\
	&\quad + \Bigg\{\left[B - \frac{\theta}{\theta +1}A\Lambda\sigma^\top(\sigma\sigma^\top)^{-1}
	+ \left(\theta\Lambda\Lambda^\top - \frac{\theta^2}{\theta +1}\Lambda\sigma^\top(\sigma\sigma^\top)^{-1}\sigma\Lambda^\top\right)\Psi_t\right]\eta_t \\
	&\qquad + b\Psi_t - \left(\frac{1}{\theta +1}A(\sigma\sigma^\top)^{-1} + \frac{\theta}{\theta +1}\Lambda\sigma^\top(\sigma\sigma^\top)^{-1}\Psi_t\right)(a - r_s) \Bigg\}x \\
	&\quad -\frac{1}{2(\theta +1)}(\sigma\sigma^\top)^{-1}\left(a - r_t + \theta\eta_t\Lambda\sigma^\top\right)^2- r_t + b\eta_t + \frac{\theta}{2}\Lambda\Lambda^\top\eta_t^2 + \frac{1}{2}\Lambda\Lambda^\top\Psi_t,
\end{aligned}
\end{equation}
provided that $\sigma\sigma^\top>0$ and $\theta>0$. Then we see that the optimal state feedback portfolio is given by
\begin{equation}\label{optimal portfolio state feedback form-DPP}\hspace{-1mm}
	\bar{u}(t,\bar{X})=\frac{1}{\theta+1}(\sigma\sigma^\top)^{-1} \left[ \big(\theta\Lambda\sigma^\top\Psi_t + A\big)\bar{X}+ \theta\Lambda\sigma^\top\eta_t + a - r_t\right],\quad \text{a.e.} s \in [t,T],\quad \hat{\mathbf{P}}\text{-a.s.}
\end{equation}
In addition, noting \eqref{optimal portfolio state feedback form-DPP}, the HJB equation \eqref{HJB equation example} now reads
\begin{equation}\label{HJB equation example 1}
\begin{aligned}
	&-\frac{1}{2}\dot{\Psi}_{t}x^2 - \dot{\eta}_{t}x - \dot{k}_{t} = \Bigg\{\frac{1}{2}\left[\theta\Lambda\Lambda^\top - \frac{\theta^2}{\theta +1}\Lambda\sigma^\top(\sigma\sigma^\top)^{-1}\sigma\Lambda^\top\right]\Psi_t^2\\
	&\quad + \left(B - \frac{\theta}{\theta +1}A\Lambda\sigma^\top(\sigma\sigma^\top)^{-1}\right)\Psi_t - \frac{1}{2(\theta +1)}A^2(\sigma\sigma^\top)^{-1}\Bigg\}x^2 \\
	&\quad + \Bigg\{\left[B - \frac{\theta}{\theta +1}A\Lambda\sigma^\top(\sigma\sigma^\top)^{-1}
	+ \left(\theta\Lambda\Lambda^\top - \frac{\theta^2}{\theta +1}\Lambda\sigma^\top(\sigma\sigma^\top)^{-1}\sigma\Lambda^\top\right)\Psi_t\right]\eta_t \\
	&\qquad + b\Psi_t - \left(\frac{1}{\theta +1}A(\sigma\sigma^\top)^{-1} + \frac{\theta}{\theta +1}\Lambda\sigma^\top(\sigma\sigma^\top)^{-1}\Psi_t\right)(a - r_s) \Bigg\}x \\
	&\quad -\frac{1}{2(\theta +1)}(\sigma\sigma^\top)^{-1}\left(a - r_t + \theta\eta_t\Lambda\sigma^\top\right)^2- r_t + b\eta_t + \frac{\theta}{2}\Lambda\Lambda^\top\eta_t^2 + \frac{1}{2}\Lambda\Lambda^\top\Psi_t.
\end{aligned}
\end{equation}
By comparing the quadratic terms and linear terms in $x$, we recover the equation \eqref{Riccati}, \eqref{ODE}, respectively. That is to say, we have that $\Psi_\cdot$ coincides with $\Gamma_\cdot$ and $\eta_\cdot$ coincides with $\phi_\cdot$. Moreover, we have
\begin{equation}\label{ODE-DPP}
\left\{
\begin{aligned}
	&\dot{k}_t-\frac{1}{2(\theta +1)}(\sigma\sigma^\top)^{-1}\left(a - r_t + \theta\eta_t\Lambda\sigma^\top\right)^2- r_t + b\eta_t + \frac{\theta}{2}\Lambda\Lambda^\top\eta_t^2 + \frac{1}{2}\Lambda\Lambda^\top\Psi_t = 0,\\
	&k_T=-\log v.
\end{aligned}
\right.
\end{equation}

We have proved the following result.
\begin{mythm}\label{DPP example}
If equations \eqref{Riccati}, \eqref{ODE} and \eqref{ODE-DPP} admit solutions $\Gamma_\cdot$, $\phi_\cdot$ and $k_\cdot$, then the optimal control of our risk-sensitive portfolio optimization problem \eqref{wealth new 1}-\eqref{value function example} has the state feedback form
\begin{equation}
	\bar{u}(t,\bar{X})=\frac{1}{\theta+1}(\sigma\sigma^\top)^{-1} \left[ \big(\theta\Lambda\sigma^\top\Gamma_t + A\big)\bar{X}
	+ \theta\Lambda\sigma^{\top}\phi_{t} + a - r_t\right],
\end{equation}
for $\text{a.e.} s \in [0,T]$, $\hat{\mathbf{P}}\text{-a.s.}$, and the value function is given by \eqref{value function example}, where $\Gamma_\cdot$, $\phi_\cdot$ and $k_\cdot$ are given by \eqref{Riccati}, \eqref{ODE} and \eqref{ODE-DPP}, respectively.
\end{mythm}

\subsection{Relationship between MP and DPP}

We now verify the relationship in Theorem \ref{relation of MP and DPP}. In fact, relationship \eqref{relation time variable} is obvious from \eqref{HJB equation example}. Moreover, \eqref{relarion p q and X example} and \eqref{conjecture of V(t,x)} are exactly the relationships given in \eqref{relation state variable}.

Let us verify the relationship \eqref{relation inequality}. We consider the second-order adjoint equation \eqref{second-order adjoint equation example}, it is easy to show that
$(\bar{P}_s,\bar{Q}_s) = (\rho_s,0)$, $s\in[0,T]$, where $\rho_\cdot$ is the solution to (noting \eqref{relarion p q and X example}):
\begin{equation}\label{rho}
\left\{
\begin{aligned}
	\dot{\rho}_s&=-2B\rho_s - \theta\Lambda\Lambda^\top \Gamma_s^2,\\
	\rho_T&=0.
\end{aligned}
\right.
\end{equation}
From \eqref{conjecture of V(t,x)}, we have $\frac{\partial^2 V}{\partial x^2}(s,\bar{X}_s^{t,x;\bar{u}}) = \Gamma_s$, where $\Gamma_\cdot$ satisfies
\begin{equation}\label{Gamma}
\left\{
\begin{aligned}
	\dot{\Gamma}_s&= - \left(\theta\Lambda\Lambda^\top - \frac{\theta^2}{\theta +1}\Lambda\sigma^\top(\sigma\sigma^\top)^{-1}\sigma\Lambda^\top\right)\Gamma_s^2 \\
	&\quad  - 2\left(B - \frac{\theta}{\theta +1}A\Lambda\sigma^\top(\sigma\sigma^\top)^{-1}\right)\Gamma_s+ \frac{1}{\theta +1}A^2(\sigma\sigma^\top)^{-1}, \\
	\Gamma_T& = 0.
\end{aligned}
\right.
\end{equation}
Furthermore, using completing of squares again, we get
\begin{equation}\label{Gamma 1}
\left\{
\begin{aligned}
	\dot{\Gamma}_s&= -2B\Gamma_s- \theta\Lambda\Lambda^\top\Gamma_s^2+ \frac{1}{\theta +1}(\sigma\sigma^\top)^{-1}\left(\theta\Lambda\sigma^\top\Gamma_s + A\right)^2 ,\\
	\Gamma_T& = 0.
\end{aligned}
\right.
\end{equation}
Noting
\begin{equation*}
	 \frac{1}{\theta +1}(\sigma\sigma^\top)^{-1}\left(\theta\Lambda\sigma^\top\Gamma_s + A\right)^2 \geq 0,
\end{equation*}
If $\frac{1}{\theta +1}(\sigma\sigma^\top)^{-1}\left(\theta\Lambda\sigma^\top\Gamma_s + A\right)^2 = 0$, then the uniqueness of the solutions to ODE guarantees that $\Gamma_s=\rho_s$, namely,
$\frac{\partial^2 V}{\partial x^2}(s,\bar{X}_s^{t,x;\bar{u}})=\bar{P}_s$. Otherwise, if $\frac{1}{\theta +1}(\sigma\sigma^\top)^{-1}\left(\theta\Lambda\sigma^\top\Gamma_s + A\right)^2 > 0$, then the comparison theorem of ODE guarantees that $\Gamma_s<\rho_s$, namely, $\frac{\partial^2 V}{\partial x^2}(s,\bar{X}_s^{t,x;\bar{u}}) < \bar{P}_s$. Thus the relationship \eqref{relation inequality} holds. And we can see the strict inequality in \eqref{relation inequality} holds in some cases.

\begin{Remark}\label{Riccati solvalbility}
If $\theta > 0$, then $\textit{I} -  \frac{\theta}{\theta +1}\sigma^\top(\sigma\sigma^\top)^{-1}\sigma$ is positive definite, then Riccati equation \eqref{Riccati} admits an explicit solution. Thus equation \eqref{ODE} admits an explicit solution (see Corollary 4.1 in Shi and Wu \cite{SW12} for details).
\end{Remark}

\section{Concluding remarks}

In this article, we have considered the relationship between MP and DPP for risk-sensitive stochastic optimal control problems. When the value function is smooth, we obtained the connection among the adjoint processes, the generalized Hamiltonian function, and the value function. As an application, a linear-quadratic risk-sensitive portfolio optimization problem in the financial market was discussed. After the state feedback optimal control was obtained by the risk-sensitive MP and the DPP respectively, the relations we obtained are illustrated explicitly.

A challenging problem is the relationship between MP and DPP without the the smooth assumption imposed on the value function. This problem may be solved in the framework of viscosity solutions (Yong and Zhou \cite{YZ99}, Nie et al. \cite{NSW17}). We will study this topic in the future.


\begin{thebibliography}{0}

	\bibitem{BJ86}E.N. Barron, R. Jensen, The Pontryagin maximum principle from dynamic programming and viscosity solutions to first-order partial differential equations, \emph{Trans. Amer. Math. Soc.}, 298(2), 635-641, 1986.

	\bibitem{BP99}T. Bielecki, S. Pliska, Risk sensitive dynamic asset management, \emph{Appl. Math. Optim.}, 39(3), 337-360, 1999.

	\bibitem{B82}A. Bensoussan, \emph{Lectures on Stochastic Control}, Lecture Notes in Math., 972, Springer-Verlag, New York, 1982.

	\bibitem{C17-1}A. Chala, Pontryagin's risk-sensitive stochastic maximum principle for backward stochastic differential equations with application, \emph{Bull. Braz. Math. Soc., New Series}, 48(3), 399-411, 2017.

	\bibitem{C17-2}A. Chala, Sufficient optimality condition for a risk-sensitive control problem for backward stochastic differential equations and an application, \emph{J. Numer. Math. Stoch.}, 9(1), 48-60, 2017.

	\bibitem{DTT15}B. Djehiche, H. Tembine, and R. Tempone, A stochastic maximum principle for risk sensitive mean-field type control, \emph{IEEE Trans. Autom. Control}, 60(10), 2640-2649, 2015.

	\bibitem{EH03}N. El Karoui, S. Hamadene, BSDEs and risk-sensitive control, zero-sum and nonzero-sum game problems of stochastic functional differential equations, \emph{Stoch. Proc. Appl.}, 107(1),
	145-169, 2003.
	\bibitem{FR75}W.H. Fleming, R.W. Rishel, \emph{Deterministic and Stochastic Optimal Control}, Springer-Verlag, New York, 1975.

	\bibitem{FS06}W.H. Fleming, H.M. Soner, \emph{Controlled Markov Processes and Viscosity Solutions}, Springer Science and Business Media, 2006.

	\bibitem{FS00}W.H. Fleming, S.J. Sheu, Risk-sensitive control and an optimal investment model, \emph{Math. Finance}, 10(2), 197-213, 2000.

    \bibitem{Hu17}M.S. Hu, Stochastic global maximum principle for optimization with recursive utilities, \emph{Probab. Uncer. Quant. Risk}, 2, Article no. 1, 2017.

	\bibitem{HJX20}M.S. Hu, S.L. Ji, and X.L. Xue, Stochastic maximum principle, dynamic programming principle, and their relationship for fully coupled forward-backward stochastic controlled systems, \emph{ESAIM Control Optim. Calc. Var.}, 26, Article no. 81, 2020.

	\bibitem{HP06}Y. Hu, S.G. Peng, On the comparision theorem for mutidimensional BSDEs, \emph{C. R. Acad. Sci. Paris, Ser I.}, 343(2), 135-140, 2006.

	\bibitem{J73}D.H. Jacobson, Optimal stochastic linear systems with exponential criteria and their relation to deterministic differential games, \emph{IEEE Trans. Autom. Control}, 18(2), 124-131, 1973.

	\bibitem{J92}M. James, Asymptotic analysis of nonlinear stochastic risk-sensitive control and differential games, \emph{Math. Cont. Sig. Syst.}, 5(4),  401-417, 1992.

	\bibitem{K00}M. Kobylanski, Backward stochastic differential equations and partial differential equations with quadratic growth, \emph{Ann. Probab.}, 28(2), 558-602, 2000.

	\bibitem{KC19}R. Khallout, A. Chala, A risk-sensitive stochastic maximum principle for fully coupled forward-backward stochastic differential equations with applications, \emph{Asian J. Control}, 22(3), 1360-1371, 2019.

	\bibitem{KN02}K. Kuroda, H. Nagai, Risk-sensitive portfolio optimization on infinite time horizon, \emph{Stoch. $\&$ Stoch. Rep}, 73(3-4), 309-331, 2002.

	\bibitem{LZ01}A.E.B. Lim, X.Y. Zhou, Risk-sensitive control with HARA utility, \emph{IEEE Trans. Autom. Control}, 46(4), 563-578, 2001.

	\bibitem{LZ05}A.E.B. Lim, X.Y. Zhou, A new risk-sensitive maximum principle, \emph{IEEE Trans. Autom. Control}, 50(7), 958-966, 2005.

	\bibitem{LS23}J.T. Lin, J.T. Shi, A global stochastic maximum principle for controlled fully coupled FBSDEs with general cost functional, In: \emph{Proc. 35th Chinese Control and Decision Conference}, 5621-5628, Yichang, China, May 20-22, 2023.

	\bibitem{LS25}J.T. Lin, J.T. Shi, A risk-sensitive global maximum principle for controlled fully coupled FBSDEs with applications, \emph{Math. Control Relat. Fields}, 15(1), 179-205, 2025.

	\bibitem{MB17}J. Moon, T. Basar, Linear quadratic risk-sensitive and robust mean field games,\emph{ IEEE Trans. Autom. Control}, 62(3), 1062-1077, 2017.

	\bibitem{MDB19}J. Moon, T.E. Ducan, and T. Basar, Risk-sensitive zero-sum differential games, \emph{IEEE Trans. Autom. Control}, 64(4), 1503-1518, 2019.

	\bibitem{M20}J. Moon, The risk-sensitive maximum principle for controlled forward-backward stochastic differential equations, \emph{Automatica}, 120, 109069, 2020.

	\bibitem{M21}J. Moon, Generalized risk-sensitive optimal control and Hamilton-Jacobi-Bellman equation, \emph{IEEE Trans. Autom. Control}, 66(5), 2319-2325, 2021.

	\bibitem{NSW17}T.Y. Nie, J.T. Shi, and Z. Wu, Connection between MP and DPP for stochastic recursive optimal control problems: Viscosity solution framework in the general case, \emph{SIAM J. Control Optim.}, 55(5), 3258-3294, 2017.

	\bibitem{P90}S.G. Peng, A general stochastic maximum principle for optimal control problems, \emph{SIAM J. Control Optim.}, 28(4), 966-979, 1990.

	\bibitem{S13}J.T. Shi, Z.Y. Yu, Relationship between maximum principle and dynamic programming for stochastic recursive optimal control problems and applications, \emph{Math. Prob. Engin.}, 2013: Article ID 285241, 12 pages, 2013.

	\bibitem{SW11}J.T. Shi, Z. Wu, A risk-sensitive stochastic maximum principle for optimal control of jump diffusions and its applications, \emph{Acta Math. Sci.}, 31B(2), 419-433, 2011.

	\bibitem{SW12}J.T. Shi, Z. Wu, Maximum principle for risk-sensitive stochastic optimal control problem and applications to finance, \emph{Stoch. Anal. Appl.}, 30(6), 997-1018, 2012.

	\bibitem{W25}J.B. Wu, B.T. Xu, and L.Q. Zhang,  Risk-sensitive singular control for stochastic recursive systems and Hamilton-Jacobi-Bellman inequality, \emph{J. Differential Equations}, 427, 641-675, 2025.

	\bibitem{W90}P. Whittle, A risk-sensitive maximum principle, \emph{Syst. $\&$ Control Lett.}, 15(3), 183-192, 1990.

	\bibitem{W91}P. Whittle, A risk-sensitive maximum principle: The case of imperfect state observation. \emph{IEEE Trans. Autom. Control}, 36(7), 793-801, 1991.

	\bibitem{YZ99}J.M. Yong, X.Y. Zhou, \emph{Stochastic Controls: Hamiltonian Systems and HJB Equations}, Springer-Verlag, New York, 1999.

	\bibitem{Z17}J.F. Zhang, \emph{Backward Stochastic Differential Equations: From Linear to Fully Nonlinear Theory}, Springer, New York, 2017.

	\bibitem{Z90-1}X.Y. Zhou, Maximum principle, dynamic programming, and their connection in determinsitc control, \emph{J. Optim. Theory Appl.}, 65(2), 363-373, 1990.

	\bibitem{Z90-2}X.Y. Zhou, The connection between the maximum principle and dynamic programming in stochastic control, \emph{Stoch. $\&$ Stoch. Rep.}, 31, 1-13, 1990.

	\bibitem{Z91}X.Y. Zhou, A unfied treatment of maximum principle and dynamic programming in stochastic controls, \emph{Stoch. $\&$ Stoch. Rep.}, 36, 137-161, 1991.

\end{thebibliography}
\end{document}